\documentclass[a4paper,12pt]{amsart}
\usepackage[T1]{fontenc}
\usepackage[centering,margin=2.3cm]{geometry}
\usepackage[latin1]{inputenc}
\usepackage{amsthm, amsmath}   
\usepackage{latexsym,amssymb}
\usepackage{algorithmic}
\usepackage{amssymb}
\usepackage{times}
\usepackage{enumerate}
\usepackage[usenames,dvipsnames]{pstricks}
\usepackage{epsfig}
\usepackage{pst-node}
\usepackage{pst-grad} 
\usepackage{pst-plot} 
\usepackage{pst-3dplot}

\newcommand{\N}{\mathbb{N}}
\newcommand{\A}{\mathbb{A}}
\newcommand{\Z}{\mathbb{Z}}

\newcommand{\mB}{\mathcal{B}}

\newcommand{\mH}{\mathcal{H}}

\newcommand{\mM}{\mathcal{M}}
\newcommand{\mK}{\mathcal{K}}

\newcommand{\mU}{\mathcal{U}}
\newcommand{\mT}{\mathcal{T}}

\newcommand{\IM}{I_{\mM}}

\newcommand{\mG}{\mathcal{G}}

\theoremstyle{plain}
\newtheorem{theorem}{Theorem}[section]
\newtheorem{lemma}[theorem]{Lemma}

\newtheorem{proposition}[theorem]{Proposition}

\theoremstyle{definition}

\newtheorem{example}[theorem]{Example}

\theoremstyle{remark}

\title{ Matroid toric ideals: complete intersection, minors and minimal systems of generators}

\address{Université de Montpellier, Institut Montpelliérain Alexander Grothendieck, Case Courrier 051, Place Eugène Bataillon, 34095 Montpellier Cedex 05, France}
\author[I. Garc\'{i}a-Marco]{Ignacio Garc\'{i}a-Marco\,*}

\thanks{* Corresponding Author: Phone: +33-624564368. Email: ignaciogarciamarco@gmail.com
\\
\indent The authors were supported by the ANR TEOMATRO grant
ANR-10-BLAN 0207.
\\
\indent The first author was partially supported by Ministerio de
Educación y Ciencia - España MTM2010-20279- C02-02.}

\email{ignaciogarciamarco@gmail.com}
\author[J.~L. Ram\'{i}rez Alfons\'{i}n]{Jorge Luis Ram\'{i}rez Alfons\'{i}n}
\email{jorge.ramirez-alfonsin@umontpellier.fr}


\date{September, 2015}

\begin{document}

\begin{abstract}
In this paper, we investigate three problems concerning the toric
ideal associated to a matroid. Firstly, we list all matroids $\mM$
such that its corresponding toric ideal $\IM$ is a complete
intersection. Secondly, we handle the problem of detecting minors of
a matroid $\mM$ from a minimal set of binomial generators of $\IM$.
In particular, given a minimal set of binomial generators of $\IM$
we provide a necessary condition for $\mM$ to have a minor
isomorphic to $\mU_{d,2d}$ for $d \geq 2$. This condition is proved
to be sufficient for $d = 2$ (leading to a criterion for determining
whether $\mM$ is binary) and for $d = 3$. Finally, we characterize
all matroids $\mM$ such that $\IM$ has a unique minimal set of
binomial generators.
\end{abstract}

\maketitle

\section{Introduction}\label{introduction}

Let $\mM$ be a matroid on a finite ground set $E = \{1,\ldots,n\}$,
we denote by $\mB$ the set of bases of $\mM$. Let $k$ be an
arbitrary field and consider $k[x_1,\ldots,x_n]$ a polynomial ring
over $k$. For each base $B \in \mB$, we introduce a variable $y_B$
and we denote by $R$ the polynomial ring in the variables $y_B$,
i.e., $R := k[y_B \, \vert \, B \in \mB]$. A {\it binomial} in $R$
is a difference of two monomials, an ideal generated by binomials is
called a {\it binomial ideal}.

\medskip We consider the homomorphism of $k$-algebras
 $ \varphi: R \longrightarrow k[x_1,\dots ,x_n]$ induced
by  $y_B  \mapsto \prod_{i \in B} x_i.$ The image of $\varphi$ is a
standard graded $k$-algebra, which is called the {\it bases monomial
ring of the matroid $M$} and it is denoted by $S_{\mM}$. By
\cite[Theorem 5]{W2}, $S_{\mM}$ has Krull dimension ${\rm
dim}(S_{\mM}) = n - c + 1$, where $c$ is the number of connected
components of $\mM$.  The number $c$ of connected components is the
largest integer $k$ such that $E$ is the disjoint union of the
nonempty sets $E_1,\ldots,E_k$ and $\mM$ is the direct sum of some
matroids $\mM_1,\ldots,\mM_k$, where $\mM_i$ has ground set $E_i$.
The kernel of $\varphi$, which is the presentation ideal of $S_M$,
is called the {\it toric ideal of $\mM$} and is denoted by $\IM$. It
is well known that $\IM$ is a prime, binomial and homogeneous ideal,
see, e.g., \cite{Sturm}. Since $R / \IM \simeq S_{\mM}$, it follows
that the height of $\IM$ is ${\rm ht}(\IM) = |\mB| - {\rm
dim}(S_{\mM})$.
\medskip

 In \cite{W1}, White posed several conjectures
concerning basis exchange properties on matroids. One of these
combinatorial conjectures turned out to be equivalent to decide if
$\IM$ is always generated by quadratics. This algebraic version of
the conjecture motivated several authors to study $\IM$. Despite
this conjecture is still open, it has been proved to be true by
means of this algebraic approach for several families of matroids
(see \cite{LasonMichalek} and the references there). Even more, it
is not even known if for every matroid its corresponding toric ideal
admits a quadratic Gr\"obner basis.
\medskip

In this paper we study the algebraic structure of toric ideals of
matroids. We study three different problems concerning $\IM$.

\subsection{Complete intersection}

The first problem is to characterize the matroids $\mM$ such that
$\IM$ is a complete intersection. The toric ideal $\IM$ is a {\it
complete intersection} if $\mu(\IM) = {\rm ht}(\IM)$, where
$\mu(\IM)$ denotes the minimal number of generators of $\IM$.
Equivalently, $\IM$ is a complete intersection if and only if there
exists a set of homogeneous binomials $g_1,\ldots,g_s \in R$ such
that $s = {\rm ht}(\IM)$ and $\IM=(g_1,\ldots,g_s)$.

\medskip

Complete intersection toric ideals were first studied by Herzog in
\cite{Herzog}. Since then, they have been extensively studied by
several authors. In the context of toric ideals associated to
combinatorial structures, the complete intersection property has
 been widely studied for graphs, see, e.g., \cite{BGRgrafos, Tatakis-Thoma, GRVega}.
In this work we address this problem in the context of toric ideals
of matroids and prove that there are essentially three matroids
whose corresponding toric ideal is a complete intersection; namely,
the rank $2$ matroids without loops or coloops on a ground set of
$4$ elements.

\subsection{Minors}

Many of the most celebrated results on matroids make reference to
minors, for this reason it is convenient to have tools to detect
whether a matroid has a certain minor or not. In this work we study
the problem of detecting whether a matroid $\mM$ has a minor
isomorphic to $\mU_{d,2d}$ with $d \geq 2$, where $\mU_{r,n}$
denotes the uniform matroid of rank $r$ on $E = \{1,\ldots,n\}$.
More precisely, we prove that whenever a matroid contains a minor
isomorphic to $\mU_{d,2d}$, then there exist $B_1, B_2 \in \mB$ such
that $\Delta_{\{B_1,B_2\}} = \binom{2d-1}{d}$; where, for every
$B_1,B_2 \in \mB$, $\Delta_{\{B_1,B_2\}}$ denotes the number of
pairs of bases $\{D_1,D_2\}$ such that $B_1 \cup B_2 = D_1 \cup D_2$
as multisets. This condition is also proved to be sufficient for $d
= 2$ and $d = 3$. Since $\mU_{2,4}$ is the only excluded minor for a
matroid to be binary, the result for $d = 2$ provides a new
criterion for detecting whether a matroid is binary. Moreover, we
provide an example to show that for $d = 5$ this condition is no
longer sufficient. These results are presented in purely
combinatorial terms, nevertheless whenever one knows a minimal set
of binomials generators of $\IM$, one can easily compute
$\Delta_{\{B_1,B_2\}}$ for all $B_1,B_2 \in \mB$. Thus, these
results give a method to detect if a matroid has a minor isomorphic
to $\mU_{2,4}$ or $\mU_{3,6}$ provided one knows a minimal set of
binomial generators of $\IM$.

\subsection{Minimal systems of generators}

Minimal systems of binomial generators of toric ideals have been
studied in several papers; see, e.g., \cite{BCMP,DS}. In general,
for a toric ideal it is possible to have more than one minimal
system of generators formed by binomials. Given a toric ideal $I$,
we denote by $\nu(I)$ the number of minimal sets of binomial
generators of $I$, where the sign of a binomial does not count. A
recent problem arising from algebraic statistics (see \cite{TA}) is
to characterize when a toric ideal $I$ possesses a unique minimal
system of binomial generators; i.e., when $\nu(I) = 1$. The problems
of determining $\nu(I)$ and characterizing when $\nu(I) = 1$ for a
toric ideal $I$ were studied in \cite{CKT, OV}, also in \cite{GO,
KO} in the context of toric ideals associated to affine monomial
curves and in \cite{OH, RTT} for toric ideals of graphs. In this
paper we also handle these problems in the context of toric ideals
of matroids. More precisely, we characterize all matroids $\mM$ such
that $\nu(\IM) = 1$. This result follows as a consequence of a lower
bound we obtain for $\nu(\IM)$. This bound turns to be an equality
whenever $\IM$ is generated by quadratics.

\medskip

 The paper is organized as follows. In the next section, we
recall how the operations of deletion and contraction on a matroid
$\mM$ reflect into $\IM$. We prove that the complete intersection
property is preserved under taking minors (Proposition \ref{minor}).
We then give a complete list of all matroids whose corresponding
toric ideal is a complete intersection (Theorem \ref{icmatroid}). To
this end, we first give such a list for matroids of rank $2$
(Proposition \ref{rango2}), which is based on results given in
\cite{BGRgrafos}. In Section \ref{secc:3}, we provide a necessary
condition for a matroid to contain a minor isomorphic to
$\mU_{d,2d}$ for $d \geq 2$ in terms of the values
$\Delta_{\{B_1,B_2\}}$ for $B_1,B_2 \in \mB$ (Proposition
\ref{minoruniforme}).  We also prove that this condition is also
sufficient when $d = 2$ or $d = 3$ (Theorems \ref{binary} and
\ref{U36minor}). Moreover, we show that this condition is no longer
sufficient for $d = 5$.  In the last section we focus on giving
formulas for the values $\mu(\IM)$ and $\nu(\IM)$. In particular, we
give a lower bound for these in terms of the values
$\Delta_{\{B_1,B_2\}}$ for $B_1,B_2 \in \mB$ (Theorem
\ref{numerosistgen}). Moreover, this lower bound turns to be exact
provided $\IM$ is generated by quadratics. Finally, we characterize
all those matroids whose toric ideal has a unique minimal binomial
generating set (Theorem \ref{unique}).

\medskip

\section{Complete intersection toric ideals of matroids}\label{sec2}

 We begin this section by setting up some notation and recalling some results about matroids which are
 useful in the sequel. For a general background on matroids we refer the reader to \cite{Oxley}.

\medskip

Let $\mM$ be a matroid on the ground set $E = \{1,\ldots,n\}$ and
rank $r$. Let $\mB$ denote the set of bases of $\mM$. By definition
$\mB$ is not empty and satisfies the following {\it exchange axiom}:
\begin{quote} For every $B_1, B_2 \in \mB$ and for every $e \in B_1 \setminus
B_2$, there exists $f \in B_2 \setminus B_1$ such that $(B_1 \cup
\{f\}) \setminus \{e\} \in \mB$. \end{quote}

\smallskip \noindent Brualdi proved in \cite{Brualdi} that the
exchange axiom is equivalent to the {\it symmetric exchange axiom}:
\begin{quote} For every $B_1, B_2$ in $\mB$ and for every $e \in B_1
\setminus B_2$, there exists $f \in B_2 \setminus B_1$ such that
both $(B_1 \cup \{f\}) \setminus \{e\} \in \mB$ and $(B_2 \cup
\{e\}) \setminus \{f\} \in \mB$.
\end{quote}

\medskip

Now we recall some basic facts and results over toric ideals of
matroids needed later on. Firstly, we observe that for
$B_1,\ldots,B_s,D_1,\ldots,D_s \in \mB$, the homogeneous binomial
$y_{B_1} \cdots y_{B_s} - y_{D_1} \cdots y_{D_s}$ belongs to $\IM$
if and only if $B_1 \cup \cdots \cup B_s = D_1 \cup \cdots \cup D_s$
as multisets. Since $\IM$ is a homogeneous binomial ideal, it
follows that $$ \IM = \big( \{y_{B_1} \cdots y_{B_s} - y_{D_1}
\cdots y_{D_s} \, \vert \, B_1 \cup \cdots \cup B_s = D_1 \cup
\cdots \cup D_s {\text\ as\ multisets } \} \big).$$ From this
expression one easily derives that whenever  $r \in \{0,1,n-1,n\}$,
then $\IM = (0)$ and $\IM$ is a complete intersection. Thus, we only
consider the case $2 \leq r \leq n - 2$.

\medskip Now we prove that the operations of taking duals, deletion, contraction and
taking minors of $\mM$ preserve the property of being a complete
intersection on $\IM$. For more details on how these operations
affect $\IM$ we refer the reader to \cite[Section 2]{Blum}.

\medskip We denote by $\mM^*$ the dual matroid of $\mM$.
It is straightforward to check that $\sigma(\IM) = I_{\mM^*}$, where
$\sigma$ is the isomorphism of $k$-algebras $\sigma: R
\longrightarrow k[y_{E \setminus B} \, \vert \, B \in \mB]$ induced
by $y_B \mapsto y_{E \setminus B}$. Thus, $\IM$ is a complete
intersection if and only if $I_{\mM^*}$ also is.

\medskip For every $A \subset E$, $\mM
\setminus A$ denotes the {\it deletion of $A$ from $\mM$} and $\mM /
A$ denotes the {\it contraction of $A$ from $\mM$}. For $E' \subset
E$, the restriction of $\mM$ to $E'$ is denoted by $\mM |_{E'}$.

\medskip

\begin{proposition}\label{minor} Let $\mM'$ be a minor of $\mM$. If
$\IM$ is a complete intersection, then $I_{\mM'}$ also is.
\end{proposition}
\begin{proof}
Take $e \in E$ and let us prove that $I_{\mM \setminus \{e\}}$ is a
complete intersection. If $e$ is a loop, then $\mB$ is the set of
bases of both $\mM$ and $\mM \setminus \{e\}$ and, hence, $\IM =
I_{\mM \setminus \{e\}}$. Assume that $e$ is not a loop and take
$\mG$ a binomial generating set of $\IM$. By \cite[Lemma
2.2]{BGRgrafos} or \cite{OHH}, $I_{\mM \setminus \{e\}}$ is
generated by the set $\mG' := \mG \cap k[y_B \, \vert \, e \notin B
\in \mB]$. Hence, $I_{\mM \setminus \{e\}}$ is a complete
intersection (see \cite[Proposition 2.3]{BGRgrafos}). An iterative
application of this result proves that for all $A \subset E$,
$I_{\mM \setminus A}$ is a complete intersection.

For every $A \subset E$, it suffices to observe that $\mM / A =
(\mM^* \setminus A)^*$ to deduce that $I_{\mM  / A}$ is also a
complete intersection whenever $\IM$  is. Thus, the result follows.
\end{proof}

\medskip As we mentioned in the proof of Proposition \ref{minor}, if $e$ is a loop then $\IM = I_{\mM \setminus \{e\}}$.
Moreover, if $e$ is a coloop of $\mM$, then $\IM$ is essentially
equal to $I_{\mM / \{e\}}$. Indeed, if one considers the isomorphism
of $k$-algebras $\tau: R \longrightarrow k[y_{B \setminus \{e\}} \,
\vert \, B \in \mB]$ induced by $y_B \mapsto y_{B \setminus \{e\}}$,
then $\tau(\IM) = I_{\mM / \{e\}}$. For this reason we may assume
without loss of generality that $\mM$ has no loops or coloops.

\medskip Now we study the complete intersection property for $\IM$
when $\mM$ has rank $2$. In this case, we associate to $\mM$ the
graph $\mH_{\mM}$ with vertex set $E$ and edge set $\mB$. It turns
out that $\IM$ coincides with the toric ideal of the graph
$\mH_{\mM}$ (see, e.g., \cite{BGRgrafos}). In particular, from
\cite[Corollary 3.9]{BGRgrafos}, we have that whenever $\IM$ is a
complete intersection, then $\mH_{\mM}$ does not contain $\mK_{2,3}$
as subgraph, where $\mK_{2,3}$ denotes the complete bipartite graph
with partitions of sizes $2$ and $3$. The following result
characterizes the complete intersection property for toric ideals of
rank $2$ matroids.

\medskip

\begin{proposition}\label{rango2}Let $\mM$ be a rank $2$ matroid on a ground set of $n \geq 4$ elements
without loops or coloops. Then, $\IM$ is a complete intersection if
and only if $n = 4$.
\end{proposition}
\begin{proof}
$(\Rightarrow)$ Assume that $n \geq 5$ and let us prove that $\IM$
is not a complete intersection. Since $\mM$ has rank $2$ and has no
loops or coloops, we may assume that it has two disjoint basis,
namely $B_1 = \{1,2\}, B_2 = \{3,4\} \in \mB.$ Moreover, $5$ is not
a coloop, so we may also assume that $B_3 = \{1,5\} \in \mB$. Since
$B_1,B_2 \in \mB$, by the symmetric exchange axiom, we can also
assume that $B_4 = \{1,3\}, B_5 = \{2,4\} \in \mB$. If $\{4,5\} \in
\mB$, then $\mH_{\mM}$ has a subgraph $\mK_{2,3}$ and $\IM$ is not a
complete intersection. Let us suppose that $\{4,5\} \notin \mB$. By
the exchange axiom for $B_2$ and $B_3$ we have $B_6 := \{3,5\} \in
\mB$. Again by the exchange axiom for $B_5$ and $B_6$ we get that
$B_7 := \{2,5\} \in \mB$. Thus, $\mH_{\mM}$ has $\mK_{2,3}$ as a
subgraph and $\IM$ is not a complete intersection.

$(\Leftarrow)$ There are three non isomorphic rank $2$ matroids
without loops or coloops and $n = 4$. Namely, $\mM_1$ with set of
bases $\mB_1 = \{\{1,2\}, \{3,4\}, \{1,3\}, \{2,4\}\}$, $\mM_2$ with
set of bases $\mB_2 = \mB_1 \cup \{\{1,4\}\}$ and $\mM_3 =
\mU_{2,4}$. For $i = 1,2$ one can easily check that ${\rm
ht}(I_{\mM_i}) = 1$ and that $I_{\mM_i} = (y_{\{1,2\}} y_{\{3,4\}} -
y_{\{1,3\}} y_{\{2,4\}})$; thus both $I_{\mM_1}$ and $I_{\mM_2}$ are
complete intersections. Moreover, ${\rm ht}(I_{\mM_3}) = 2$ and a
direct computation with {\sc Singular} \cite{DGPS} or {\sc CoCoA}
\cite{ABL} yields that $I_{\mM_3} = (y_{\{1,2\}} y_{\{3,4\}} -
y_{\{1,3\}} y_{\{2,4\}}, y_{\{1,4\}} y_{\{2,3\}} - y_{\{1,3\}}
y_{\{2,4\}})$; thus $I_{\mM_3}$ is also a complete intersection.
\end{proof}

\medskip

Now, we apply Proposition \ref{rango2} to give the list of all
matroids $\mM$ such that $\IM$ is a complete intersection.

\begin{theorem}\label{icmatroid}Let $\mM$ be a matroid without loops
or coloops and with $2 \leq r \leq n - 1$. Then, $\IM$ is a complete
intersection if and only if $n = 4$ and $\mM$ is the matroid whose
set of bases is:
\begin{enumerate}
\item $\mB = \{\{1,2\},\{3,4\},\{1,3\},\{2,4\}\},$
\item $\mB = \{\{1,2\},\{3,4\},\{1,3\},\{2,4\},
\{1,4\}\},$ or
\item $\mB = \{\{1,2\},\{3,4\},\{1,3\},\{2,4\},
\{1,4\}, \{2,3\}\}$, i.e., $\mM = \mU_{2,4}$.
\end{enumerate}
\end{theorem}
\begin{proof}By Proposition \ref{rango2} it only remains to prove that $\IM$ is not a complete
intersection provided $r \geq 3$. Since $n > r + 1$ and $\mM$ has no
loops or coloops, we can take $B_1, B_2 \in \mB$ such that $|B_1
\setminus B_2| = 2$ and consider $f \in B_1 \cap B_2$. Since $f$ is
not a coloop, there exists $B' \in \mB$ such that $f \notin B'$.
Moreover, since $B_1, B' \in \mB$, by the exchange axiom there
exists $e \in B'$ such that $B_3 := (B_1 \setminus \{f\}) \cup \{e\}
\in \mB$. We observe that $|B_2 \setminus B_3| \in \{2,3\}$. Setting
$A := B_1 \cap B_2 \cap B_2$, we can assume without loss of
generality that  $f = 1$ and that  $B_1 = A \cup \{1,2,3\}$, $B_2 =
A \cup \{1,4,5\}$ and $B_3 = A \cup \{2,3,e\}$, where $e \in
\{5,6\}$. We have two cases.
\smallskip

{\it Case 1: $e = 5$}. We consider the matroid $(\mM')^*$, the dual
matroid of $\mM' := (\mM / A) | E'$, with $E' = \{1,2,3,4,5\}$. We
observe that $\{1,2,3\}, \{1,4,5\}, \{2,3,5\}$ are bases of $\mM'$
and hence $\{4,5\}, \{2,3\}, \{1,4\}$ are bases of $(\mM')^*$. Thus
$(\mM')^*$ is a rank $2$ matroid without loops or coloops and, by
Proposition \ref{rango2}, $I_{(\mM')^*}$ is not a complete
intersection. Hence, by Proposition \ref{minor}, we conclude that
$\IM$ is not a complete intersection.
\smallskip

{\it Case 2: $e = 6$}. We consider the minor $\mM' := (\mM / A) |
E'$, where $E' = \{1,2,3,4,5,6\}$ and observe that $\{1,4,5\},
\{1,2,3\}, \{2,3,6\}$ are bases of $\mM'$. By the symmetric exchange
axiom, we may also assume that $\{1,2,4\}, \{1,3,5\}$ are also bases
of $\mM'$. We claim that for every base $B$ of $\mM$, either $1 \in
B$ or $6 \in B$, but not both. Indeed, if there exists a base $B$ of
$\mM'$ such that $\{1,6\} \subset B$ then the rank $2$ matroid
$\mM_1 := \mM' / \{1\}$ on the set $E' \setminus \{1\}$ has no loops
or coloops. Thus, by Proposition \ref{rango2}, $I_{\mM_1}$ is not a
complete intersection and, by Proposition \ref{minor}, neither is
$\IM$. If there exists a base of $\mM'$ such that $1 \notin B$ and
$6 \notin B$, the rank $2$ matroid $\mM_2 := (\mM' \setminus
\{6\})^*$ on the set $E' \setminus \{6\}$ has no loops or coloops.
Thus again by Proposition \ref{rango2}, we get that $I_{\mM_1}$ is
not a complete intersection and, by Proposition \ref{minor}, neither
is $\IM$. Analogously, one can prove that for every base $B$ of
$\mM'$ either $2 \in B$ or $5 \in B$ but not both, and that either
$3 \in B$ or $4 \in B$ but not both. Hence, $\mM'$ is the
transversal matroid with presentation $(\{1,6\}, \{2,5\} ,\{3,4\})$.
Since $\mM'$ has $8$ bases and $3$ connected components, then
$I_{\mM'}$ has height $4$. Moreover, a direct computation yields
that $I_{\mM'}$ is minimally generated by $9$ binomials; thus,
$I_{\mM'}$ is not a complete intersection and the proof is finished.
\end{proof}

\section{Finding minors in a matroid}\label{secc:3}

In this section we investigate a characterization for a matroid to
contain certain minors in terms of a set of binomial generators of
its corresponding toric ideal. In particular, we focus our attention
to detect if a matroid $\mM$ contains a minor $\mU_{d,2d}$ for $d
\geq 2$. We consider the following binary equivalence relation
$\sim$ on the set of pairs of bases:

\begin{center} $\{B_1,B_2\} \sim \{B_3, B_4\} \ \Longleftrightarrow
\ B_1 \cup B_2 = B_3 \cup B_4$ as multisets,
\end{center} and we denote by $\Delta_{\{B_1,B_2\}}$ the cardinality of the equivalence
class of $\{B_1,B_2\}$.

\medskip For two sets $A,B$ we denote by $A \bigtriangleup B$ the
{\it symmetric difference} of $A$ and $B$, i.e., $A \bigtriangleup B
:= (A \setminus B) \cup (B \setminus A).$

\medskip We now introduce two lemmas concerning the values $\Delta_{\{B_1,B_2\}}$.
The first one provides some bounds on the values of
$\Delta_{\{B_1,B_2\}}$. In the proof of this lemma we use the so
called {\it multiple symmetric exchange property} (see
\cite{Woodall}):
\begin{quote}For every $B_1, B_2$ in $\mB$ and for every $A_1 \subset B_1$, there exists $A_2 \subset B_2$ such that
$(B_1 \cup A_2) \setminus A_1 \in \mB$ and $(B_2 \cup A_1) \setminus
A_2$ are in $\mB$.
\end{quote}

\medskip

\begin{lemma}\label{cotas}For every $B_1, B_2 \in \mB$, then
$2^{d-1} \leq \Delta_{\{B_1,B_2\}} \leq \binom{2d-1}{d},$ where $d
:= | B_1 \setminus B_2 |$.
\end{lemma}
\begin{proof}Take $e \in B_1 \setminus B_2$. By the multiple symmetric exchange
property, for every $A_1$ such that $e \in A_1 \subset (B_1
\setminus B_2)$, there exists $A_2 \subset B_2$ such that both $B_1'
:= (B_1 \cup A_2) \setminus A_1$ and $B_2' := (B_2 \cup A_1)
\setminus A_2$ are bases. Since $B_1 \cup B_2 = B_1' \cup B_2'$ as
multisets, we derive that $\Delta_{\{B_1,B_2\}}$ is greater or equal
to the number of sets $A_1$ such that $e \in A_1 \subset (B_1
\setminus B_2)$, which is exactly $2^{d-1}$.

We set $A := B_1 \cap B_2$, $C := B_1 \bigtriangleup B_2$ and take
$e \in B_1 \setminus B_2$. Take $B_3,B_4 \in \mB$ such that $B_1
\cup B_2 = B_3 \cup B_4$ as multisets and assume that $e \in B_4$.
Then, $B_3 \setminus A \subset C \setminus \{e\}$ with $|B_3
\setminus A|=|B_1\setminus B_2|=d$ elements; thus,
$\Delta_{\{B_1,B_2\}} \leq \binom{2d-1}{d}$.
\end{proof}

\medskip

Moreover, the bounds of Lemma \ref{cotas} are sharp for every $d
\geq 2$. Indeed, if one considers the transversal matroid on the set
$\{1,\ldots,2d\}$ with presentation $(\{1,d+1\},\ldots,\{d,2d\})$,
and takes the bases $B_1 = \{1,\ldots,d\}$, $B_2 =
\{d+1,\ldots,2d\}$, then $|B_1 \setminus B_2| = d$ and
$\Delta_{\{B_1,B_2\}} = 2^{d-1}$. Also, if we consider the uniform
matroid $\mU_{d,2d}$ then for any base $B$ we have that
$\Delta_{\{B, E \setminus B\}} = \binom{2d-1}{d}$.

\medskip

The second lemma interprets the values of $\Delta_{\{B_1,B_2\}}$ in
terms of the number of bases-cobases of a certain minor of $\mM$.
Recall that a base $B \in \mB$ is a {\em base-cobase} if $E
\setminus B$ is also a base of $\mM$.

\medskip

\begin{lemma}\label{basecobase}Let $B_1,B_2 \in \mB$ of a matroid $\mM$ and
consider the matroid $\mM' := (\mM / (B_1 \cap B_2))|_{(B_1
\bigtriangleup B_2)}$ on the ground set $B_1 \bigtriangleup B_2$.
Then, the number of bases-cobases of $\mM'$ is equal to $2
\Delta_{\{B_1,B_2\}}$.
\end{lemma}
\begin{proof}Set $t := \Delta_{\{B_1,B_2\}}$ and consider
$B_3,B_4,\ldots,B_{2t} \in \mB$ such that $B_1 \cup B_2 = B_{2i-1}
\cup B_{2i}$ as multisets for all $i \in \{1,\ldots,t\}$. Take $i
\in \{1,\ldots,t\}$, then $B_1 \cap B_2 \subset B_{2i-1}, B_{2i}
\subset B_1 \cup B_2$ and, thus, $B_{2i-1} \setminus (B_1 \cap B_2)$
and $B_{2i} \setminus (B_1 \cap B_2)$ are complementary
bases-cobases of $\mM'$. This proves that $2 t$ is less or equal to
the number of bases-cobases of $\mM'$

Conversely, take $D_1'$ a base-cobase of $\mM'$ and denote by $D_2'$
its complementary base-cobase of $\mM'$, i.e., $D_1' \cup D_2' = B_1
\bigtriangleup B_2$. Moreover, if we set $D_i := D_i' \cup (B_1 \cap
B_2) \in \mB$ for $i = 1,2$, then $D_1 \cup D_2 = B_1 \cup B_2$ as
multisets. This proves that $2 t$ is greater or equal to the number
of bases-cobases of $\mM'$.
\end{proof}

\medskip

The following result provides a necessary condition for a matroid to
have a minor isomorphic to $\mU_{d,2d}$.

\begin{proposition}\label{minoruniforme}
If $\mM$ has a minor $\mM' \simeq \mU_{d,2d}$ for some $d \geq 2$,
then there exist $B_1,B_2\in \mB$ such that $\Delta_{\{B_1,B_2\}} =
\binom{2d-1}{d}$.
\end{proposition}
\begin{proof}Let $A,C \subset E$ be disjoint sets such that
$\mM' := (\mM \setminus A) / C \simeq \mU_{d,2d}$ and denote $E' :=
E \setminus (A \cup C)$. Since $\mM' = (\mM \setminus A) / C$, then
there exist  $e_1,\ldots, e_{r-d} \in A \cup C$ such that $B' \cup
\{e_1,\ldots,e_{r-d}\} \in \mB$ for every $B'$ base of $\mM'$
 (notice that the set $\{e_1,...e_{r-d}\}$ might not
only have elements of $C$). We take any $D \subset E'$ with $d$
elements, we have that $B_1 = D \cup \{e_1,\ldots,e_{r-d}\} \in
\mB$, $B_2 = (E' \setminus D) \cup \{e_1,\ldots,e_{r-d}\} \in \mB$
and $B_1 \cup B_2 = E' \cup \{e_1,\ldots,e_{r-d}\}$. Thus,
$\Delta_{\{B_1,B_2\}} \geq \binom{2d}{d} / 2 = \binom{2d-1}{d}.$
Since $|B_1 \setminus B_2| = d$, by Lemma \ref{cotas} we are done.
\end{proof}

\medskip
Since $\mU_{2,4}$ is the only forbidden minor for a matroid to be
binary, (see, e.g., \cite[Theorem 6.5.4]{Oxley}) the following
result gives a criterion for $\mM$ to be binary by proving the
converse of Proposition \ref{minoruniforme} for $d = 2$.

\begin{theorem}\label{binary}
$\mM$ is binary if and only if $\Delta_{\{B_1,B_2\}} \neq 3$ for
every $B_1,B_2\in \mB$.
\end{theorem}
\begin{proof}$(\Rightarrow)$ Assume that there exist $B_1,
B_2 \in \mB$ such that $\Delta_{\{B_1,B_2\}} = 3$. Let us denote $d
:= |B_1 \setminus B_2|$. By Lemma \ref{cotas} we observe that $d =
2$. If we set $C := B_1 \cap B_2$ and $A = E \setminus (B_1 \cup
B_2)$, then $\mM' := (\mM \setminus A) / C $ is a rank $2$ matroid
on a ground set of $4$ elements and, by Lemma \ref{basecobase}, it
has $6$ bases-cobases, thus $\mM' \simeq \mU_{2,4}$ and $\mM$ is not
binary.

$(\Leftarrow)$ Assume that $\mM$ is not binary, then $\mM$ has a
minor $\mM' \simeq \mU_{2,4}$ and the result follows from
Proposition \ref{minoruniforme}.
\end{proof}

\medskip
We also prove that the converse of Proposition \ref{minoruniforme}
also holds for $d = 3$. In order to prove this we make use of the
database of matroids available at \begin{center} {\tt
www-imai.is.s.u-tokyo.ac.jp/$\sim$ymatsu/matroid/index.html}
\end{center} which is based on \cite{MMIB}. This database includes
all matroids with $n \leq 9$ and all matroids with $n = 10$ and $r
\neq 5$.

\medskip

\begin{theorem}\label{U36minor}
$\mM$ has a minor $\mM' \simeq \mU_{3,6}$ if and only if
$\Delta_{\{B_1,B_2\}} = 10$ for some $B_1,B_2\in \mB$.
\end{theorem}
\begin{proof}
$(\Rightarrow)$ It follows from Proposition \ref{minoruniforme}.
\smallskip

$(\Leftarrow)$ Assume that there exist $B_1, B_2 \in \mB$ such that
$\Delta_{\{B_1,B_2\}} = 10$. We denote $d := |B_1 \setminus B_2|$
and, by Lemma \ref{cotas}, we observe that $d \in \{3,4\}$. We set
$C := B_1 \cap B_2$, $A = E \setminus (B_1 \cup B_2)$ and $\mM' :=
(\mM \setminus A) / C$, the rank $d$ matroid on the ground set $E' =
(B_1 \cup B_2) \setminus C$ with $2d$ elements. Moreover, by Lemma
\ref{basecobase}, $\mM'$ has exactly 20 bases-cobases.  An
exhaustive computer aided search among the $940$ non-isomorphic rank
$4$ matroids on a set of $8$ elements proves that there does not
exist such a matroid.  Therefore $d = 3$, and $\mM'$ is a rank $3$
matroid on a ground set of $6$ elements with $20$ bases-cobases,
thus $\mM' \simeq \mU_{3,6}$.
\end{proof}

 In view of Theorems \ref{binary} and
\ref{U36minor}, one might wonder if the condition
$\Delta_{\{B_1,B_2\}} = \binom{2d-1}{d}$ for some $B_1,B_2\in \mB$
is also sufficient to have $\mU_{d,2d}$ as a minor. For $d = 4$, we
do not know it is true or not. Nevertheless, Example \ref{noU510}
shows that for $d = 5$ this is no longer true. That is to say, there
exists a matroid $\mM$ with two bases $B_1,B_2$ such that
$\Delta_{\{B_1,B_2\}} = \binom{9}{5} = 126$ and $\mM$ has not a
minor isomorphic to $\mU_{5,10}$. To prove this result we use the
fact that there exist rank $3$ matroids with exactly $k$
bases-cobases for $k = 14$ and for $k = 18$. We have found these
matroids by an exhaustive search among the $36$ non-isomorphic
matroids of rank $3$ on a set of $6$ elements.

\medskip

\begin{example}\label{noU510}Let $\mM_1, \mM_2$ be rank $3$ matroids on the sets $E_1$ and $E_2$ with
 exactly $14$ and $18$ bases-cobases respectively. Consider the matroid $\mM := \mM_1 \oplus \mM_2$, i.e., the direct
sum of $\mM_1$ and $\mM_2$. It is easy to check that $\mM$ has
exactly $14 \cdot 18 = 252$ bases-cobases. Therefore, if we take $B$
a base-cobase of $\mM$ and denote by $B'$ its complementary
base-cobase, then $\Delta_{\{B,B'\}} = 252 / 2 = 126$. Let us see
now that $\mM$ has not a minor isomorphic to $\mU_{5,10}$. Suppose
that there exist $A,B \subset E_1 \cup E_2$ such that $\mU_{5,10}
\simeq (\mM \setminus A) / B$. We observe that $A \cup B$ has two
elements and if we denote $A_i := A \cap E_i$ and $B_i := B \cap
E_i$ for $i = 1,2$, then $\mU_{5,10} \simeq (\mM \setminus A) / B =
((\mM_1 \setminus A_1) / B_1) \oplus ((\mM_2 \setminus A_2) / B_2)$,
but this is not possible since $\mU_{5,10}$ has only one connected
component.
\end{example}

\medskip
 One of the interests in Proposition \ref{minoruniforme} and
Theorems \ref{binary} and \ref{U36minor} comes from the fact that
 for every $B_1,B_2 \in \mB$, the values of
$\Delta_{\{B_1,B_2\}}$ can be directly computed from a minimal set
of generators of $\IM$ formed by binomials. The following
proposition can be obtained as a consequence of  \cite[Theorems 2.5
and 2.6]{CKT}. However, we find it convenient to include a direct
proof of this fact.

\begin{proposition}\label{valordelta} Let $\{g_1,\ldots,g_s\}$ be a minimal
set of binomial generators of $\IM$. Then,
\begin{center}$\Delta_{\{B_1,B_2\}} = 1 + |\{g_i = y_{B_{i_1}} y_{B_{i_2}} -
y_{B_{i_3}} y_{B_{i_4}} \, \vert \,$ $B_{i_1} \cup B_{i_2} = B_1
\cup B_2$ as a multiset$\}|$\end{center} for every $B_1,B_2 \in
\mB$.
\end{proposition}
\begin{proof}Set $\mH := \{g_1,\ldots,g_s\}$ and take $B_1,B_2 \in \mB$.
Assume that $g_1,\ldots,g_t \in \mH$ are of the form $g_i =
y_{B_{i_1}} y_{B_{i_2}} - y_{B_{i_3}} y_{B_{i_4}}$ with $B_{i_1}
\cup B_{i_2} = B_1 \cup B_2$ as a multiset. We consider the graph
$\mG$ with vertices $\{B_j,B_k\} \subset \mB$ such that $B_j \cup
B_k = B_1 \cup B_2$ as multisets and, for every $i \in
\{1,\ldots,t\}$, if $g_i = y_{B_{i_1}} y_{B_{i_2}} - y_{B_{i_3}}
y_{B_{i_4}}$ then $f_i$ is the edge connecting $\{B_{i_1},B_{i_2}\}$
and $\{B_{i_3},B_{i_4}\}$. We observe that $\mG$ has
$\Delta_{\{B_1,B_2\}}$ vertices and $t$ edges; to conclude that
$\Delta_{\{B_1,B_2\}} = t + 1$ we prove that $\mG$ is a tree. Assume
that $\mG$ has a cycle and suppose that the sequence of edges
$(f_1,\ldots,f_k)$ forms a cycle. After replacing $g_i$ by $-g_i$ if
necessary, we get that $g_1 + \cdots + g_k = 0$, which contradicts
the minimality of $\mH$. Assume now that $\mG$ is not connected and
denote by $\mG_1$ one of its connected components. We take
$\{B_{j_1},B_{j_2}\}$ a vertex of $\mG_1$, $\{B_{k_1},B_{k_2}\}$ a
vertex which is not in $\mG_1$ and consider $q := y_{B_{j_1}}
y_{B_{j_2}} - y_{B_{k_1}} y_{B_{k_2}} \in \IM$. We claim that $q$
can be written as a combination of $g_1,\ldots,g_t$, i.e., $q =
\sum_{i = 1}^t q_i g_i$ for some $q_1,\ldots,q_t \in R$. Indeed, the
matroid $\mM$ induces a grading on $R$ by assigning to $y_B$ the
degree ${\rm deg}_{\mM}(y_B) := \sum_{i \in \mB} e_i \in \N^n$,
where $\{e_1,\ldots,e_n\}$ is the canonical basis of $\Z^n$.  Since
$\IM$ is a graded ideal with respect to this grading, whenever $q
\in \IM$ one may assume that $q$ can be written as a combination of
the $g_i$ such that ${\rm deg}_{\mM}(g_i)$ is componentwise less or
equal to ${\rm deg}_{\mM}(q)$. By construction of $q$, we have that
${\rm deg}_{\mM}(g_i)$ is componentwise less or equal to ${\rm
deg}_{\mM}(q)$ if and only if $i \in \{1,\ldots,t\}$ and the claim
is proved. Moreover, if we consider $\mB_1 := \cup_{\{B,B'\} \in
V(G_1)} \{B, B'\}$ and the homomorphism of $k$-algebras $\rho: R
\rightarrow k[y_B \, \vert \, B \in \mB_1]$ induced by $y_B \mapsto
y_B$ if $B \in \mB_1$, or $y_B \mapsto 0$ otherwise, then
$y_{B_{j_1}} y_{B_{j_2}} = \rho(q) = \sum_{f_i \in E(\mG_1)}
\rho(q_i) g_i$, which is not possible. Thus, we conclude that $\mG$
is connected and that $\Delta_{\{B_1,B_2\}} = t + 1$.
\end{proof}

\medskip

\section{Matroids with a unique set of binomial generators}

In general, for a toric ideal it is possible to have more than one
minimal system of generators formed by binomials. For example, as we
saw in the proof of Proposition \ref{rango2}, the matroid
$\mU_{2,4}$ is minimally generated by $\{f_1,f_2\}$, where $f_1 :=
y_{\{1,2\}} y_{\{3,4\}} - y_{\{1,3\}} y_{\{2,4\}}$ and $f_2 :=
y_{\{1,4\}} y_{\{2,3\}} - y_{\{1,3\}} y_{\{2,4\}}$; nevertheless, if
we consider $f_3 := y_{\{1,2\}} y_{\{3,4\}} - y_{\{1,4\}}
y_{\{2,3\}}$ one can easily check that $\IM$ is also minimally
generated by $\{f_1,f_3\}$ and by $\{f_2,f_3\}$. Thus,
$\mu(I_{\mU_{2,4}}) = 2$ and $\nu(I_{\mU_{2,4}}) \geq 3$.

\medskip In this section we begin by giving some bounds for the values of
$\mu(\IM)$ and $\nu(\IM)$ in terms of the values
$\Delta_{\{B_1,B_2\}}$ for $B_1,B_2 \in \mB$. Moreover, this lower
bounds turn out to be the exact values if $\IM$ is generated by
quadratics.

\begin{theorem}\label{numerosistgen}Let $R = \{ \{B_1,B_2\},\ldots,\{B_{2s-1}, B_{2s}\}\}$ be a set of
representatives of $\sim$ and set $r_i :=
\Delta_{\{B_{2i-1},B_{2i}\}}$ for all $i \in \{1,\ldots,s\}$. Then,
\begin{enumerate} \item $\mu(\IM) \geq  (b^2 - b - 2s)/ 2$, where $b := |\mB|$, and \item $\nu(\IM) \geq \prod_{i = 1}^s r_i^{\, r_i - 2}.$
\end{enumerate}
Moreover, in both cases equality holds whenever $\IM$ is generated
by quadratics.
\end{theorem}
\begin{proof}From Proposition \ref{valordelta}, we deduce that $\mu(\IM) \geq \sum_{i = 1}^s
(\Delta_{\{B_{2i-1},B_{2i}\}} - 1)$ with equality if and only if
$\IM$ is generated by quadratics. It suffices to observe that
$\sum_{i = 1}^s \Delta_{\{B_{2i-1},B_{2i}\}} = b (b-1) / 2$ to prove
{\it (1)}.

For each $i \in \{1,\ldots,s\}$ we consider the complete graph
$\mG_i$ with vertices $\{B_{j_1},B_{j_2}\}$ such that $B_{2i-1} \cup
B_{2i} = B_{j_1} \cup B_{j_2}$ as multiset. We consider $\mT_i$ a
spanning tree of $\mG$ and define $\mH_i := \{y_{B_{j_1}}y_{B_{j_2}}
- y_{B_{j_3}} y_{B_{j_4}} \, \vert \,$ the vertices $
\{B_{j_1},B_{j_2}\}$ and $\{B_{j_3}, B_{j_4}\}$ are connected by an
edge in $\mT_i\}$ and $\mH := \cup_{i = 1}^s \mH_i$. Since $\mH$ is
formed by degree $2$ polynomials which are $k$-linearly independent,
then $\mH$ can be extended to a minimal set of generators of $\IM$.
Since $\mG_i$ has exactly $r_i$ vertices, then there are exactly
$r_i^{\, r_i-2}$ different spanning trees of $\mG_i$ that lead to
different minimal systems of generators and, thus, $\nu(\IM) \geq
\prod_{i = 1}^s r_i^{\, r_i - 2}$. Moreover, if $\IM$ is generated
by quadratics, let us see that the set $\mH$ is a set of generators
itself. Indeed, let $f \in \IM$ be a binomial of degree two, then $f
= y_{B_{k_1}} y_{B_{k_2}} - y_{B_{k_3}} y_{B_{k_4}}$. We take $i \in
\{1,\ldots,s\}$ such that $\{B_{k_1},B_{k_2}\} \simeq
\{B_{k_3},B_{k_4}\} \simeq \{B_{2i-1},B_{2i}\}$ and there exists a
path in $\mT_i$ connecting the vertices $\{B_{k_1},B_{k_2}\}$ and
$\{B_{k_3},B_{k_4}\}$, the edges in this path correspond to
binomials in $\mH$ and $f$ is a combination of these binomials.
\end{proof}

\medskip

 We end by characterizing all matroids whose
toric ideal has a unique minimal binomial generating set. We recall
that the {\em basis graph of a matroid $\mM$} is the undirected
graph $\mG_{\mM}$ with vertex set $\mB$ and edges $\{B,B'\}$ such
that $|B \setminus B'| = 1.$ We also recall that the {\em diameter
of a graph} is the maximum distance between two vertices of the
graph.

\medskip

\begin{theorem}\label{unique} Let $\mM$ be a rank $r \geq 2$ matroid. Then, $\nu(\IM) = 1$ if and only if
$\mM$ is binary and the diameter of $\mG_{\mM}$ is at most $2$.
\end{theorem}
\begin{proof}
$(\Rightarrow)$ By Theorem \ref{numerosistgen},we have that
$\Delta_{\{B_1,B_2\}} \in \{1,2\}$ for all $B_1, B_2 \in \mB$. By
Lemma \ref{cotas} and Theorem \ref{binary}, this is equivalent to
$\mM$ is binary and $|B_1 \setminus B_2| \in \{1,2\}$ for all $B_1,
B_2 \in \mB$. Clearly this implies that the diameter of $\mG_{\mM}$
is less or equal to $2$.
\smallskip

$(\Leftarrow)$ Assume that the diameter of $\mG_{\mM}$ is $\leq 2$,
we claim that $\mM$ is strongly base orderable. Recall that a
matroid is strongly base orderable if for any two bases $B_1$ and
$B_2$ there is a bijection $\pi: B_1 \rightarrow B_2$ such that
$(B_1 \setminus C) \cup \pi(C)$ is a basis for all $C \subset B_1$.
We take $B_1, B_2 \in \mB$ and observe that $|B_1 \setminus B_2| \in
\{1,2\}$. If $B_1 \setminus B_2 = \{e\}$ and $B_2 \setminus B_1 =
\{f\}$ if suffices to consider the bijection $\pi: B_1 \rightarrow
B_2$ which is the identity on $B_1 \cap B_2$ and $\pi(e) = f$.
Moreover, if $B_1 \setminus B_2 = \{e_1,e_2\}$ and $B_2 \setminus
B_1 = \{f_1,f_2\}$, we denote $A := B_1 \cap B_2$ and, by the
symmetric exchange axiom, we can assume that both $A \cup
\{e_1,f_1\}$ and $A \cup \{e_2,f_2\}$ are basis of $\mM$; then it
suffices to
 consider $\pi: B_1 \rightarrow B_2$ the
identity on $A$, $\pi(e_1) = f_2$ and $\pi(e_2) = f_1$ to conclude
that $\mM$ is strongly base orderable. So, by \cite[Theorem
2]{LasonMichalek}, $\IM$ is generated by quadratics. Moreover, from
Lemma \ref{cotas} and Theorem \ref{binary} we deduce that
$\Delta_{\{B_1,B_2\}} \in \{1,2\}$ for all $B_1, B_2 \in \mB$.
Hence, the result follows by Theorem \ref{numerosistgen}.
\end{proof}

\bibliographystyle{plain}

\end{document}